\documentclass[10pt]{amsart}
\usepackage{amssymb}
\usepackage{amsmath}
\newtheorem{thm}{Theorem}[section]
\newtheorem{prop}[thm]{Proposition}
\newtheorem{cor}[thm]{Corollary}
\newtheorem{lem}[thm]{Lemma}
\newtheorem{rem}[thm]{Remark}

\newcommand{\C}{{\mathcal C}}
\newcommand{\D}{{\mathcal D}}
\newcommand{\Z}{{\mathcal Z}}
\newcommand{\E}{{\mathcal E}}
\newcommand\rep{\operatorname{rep}}
\newcommand\Irr{\operatorname{Irr}}
\newcommand\FPdim{\operatorname{FPdim}}

\newcommand\vect{\operatorname{Vec}}

\newcommand\ad{\operatorname{ad}}

\newcommand\pt{\operatorname{pt}}
\newcommand\Hom{\operatorname{Hom}}

\newcommand\Pic{\operatorname{G}}
\newcommand{\1}{\textbf{1}}

\begin{document}

\title[Fusion categories of small integral dimension]{Frobenius property for fusion categories of small integral dimension}
\author{Jingcheng Dong}
\author{Sonia Natale}
\author{Leandro Vendramin}



\date{April 22, 2014}

\address{Jingcheng Dong: College of Engineering, Nanjing Agricultural University, Nanjing 210031, P.R.
China} \email{dongjc@njau.edu.cn}
\address{Sonia Natale: Facultad de Matem\'atica, Astronom\'\i a y F\'\i sica.
Universidad Nacional de C\'ordoba. CIEM-CONICET. Ciudad
Universitaria. (5000) C\'ordoba, Argentina}
\email{natale@famaf.unc.edu.ar}
\urladdr{http://www.famaf.unc.edu.ar/~natale}

\address{Leandro Vendramin:
Philipps-Universit\" at Marburg,
Fachbereich Mathematik und Informatik,
Hans-Meerwein-Strasse,
D-35032 Marburg,
Germany
}
\email{lvendramin@dm.uba.ar}
\urladdr{http://mate.dm.uba.ar/~lvendram}

\thanks{The research of J. Dong was partially
supported by the Natural Science Foundation of China (11201231), the China
postdoctoral science foundation (2012M511643) and the Jiangsu planned projects
for postdoctoral research funds (1102041C). The work of S. Natale was partially
supported by CONICET and Secyt-UNC. L. Vendramin was supported by CONICET and
the Alexander von Humboldt Foundation}

\begin{abstract}
Let $k$ be an algebraically closed field of characteristic zero.
In this paper we prove that fusion categories of
Frobenius-Perron dimensions $84$ and $90$ are of Frobenius type.
Combining this with previous results in the literature, we obtain
that every weakly integral fusion category of Frobenius-Perron
dimension less than $120$ is of Frobenius type.
\end{abstract}
 \maketitle



\section{Introduction and main results}\label{sec1}

A fusion category $\C$ is called of \emph{Frobenius type} if for
every simple object $X$ of $\C$, the Frobenius-Perron dimension of
$X$ divides the Frobenius-Perron dimension of $\C$, that is, the
ratio $\FPdim \C/\FPdim X$ is an algebraic integer. A classical result of Frobenius asserts that if $\C$ is the category of finite-dimensional
representations of a finite group, then $\C$ is of Frobenius type.
 Kaplansky conjectured that the representation category
of every finite dimensional semisimple Hopf algebra is of
Frobenius type; this is the sixth one of ten questions posed by
Kaplansky in 1975 \cite[Appendix 2]{Kaplansky}. In general, the
conjecture is still open.

A more general related question was raised in \cite[Question
1]{ENO2}: it was asked in that paper if there exists a fusion category $\C$
which does not have the \emph{strong Frobenius property}, that is,
such that there is a simple object in an indecompo\-sable module
category over $\C$ whose Frobenius-Perron dimension does not
divide the Frobenius-Perron dimension of  $\C$. Recall from
\cite{ENO2} that a fusion category $\C$ is called \emph{weakly
group-theoretical} if it is tensor Morita equivalent to a
nilpotent fusion category. By \cite[Theorem 1.5]{ENO2} every
weakly group-theoretical fusion category does have the strong
Frobenius property. Then every such fusion category is
of Frobenius type.

\medbreak In Theorems \ref{thm2} and \ref{thm3} we study types for integral fusion categories
of Frobenius-Perron dimensions 84 and 90. Combining this with the results
of the paper \cite{ENO2} for fusion categories of Frobenius-Perron
dimensions $p^aq^b$, $pqr$ and $60$, we obtain the following
theorem.

\begin{thm}\label{fpdim100} Let $\C$ be a fusion category of integer
Frobenius-Perron dimension less than $120$. Then $\C$ is of
Frobenius type. Furthermore,  if $\FPdim \C > 1$ and $\C \ncong
\rep \mathbb A_5$, then $\C$ has nontrivial invertible objects.
\end{thm}

Theorem \ref{fpdim100} will be proved in Section \ref{sec3}.
Recall that, according to the definition given in \cite{ENO2}, a fusion category
is called \emph{simple} if it has no nontrivial proper fusion
subcategories.  As a consequence of Theorem \ref{fpdim100}, if
$\FPdim \C \leq 119$ and $\FPdim \C \neq 60$ or $p$, where $p$ is
a prime number, then $\C$ is not simple as a fusion
category. Combined with the results of the paper \cite{ENO2}, the
theorem implies that the only weakly integral simple fusion
categories of Frobenius-Perron dimension $\leq 119$ are the
categories $\rep \mathbb A_5$ of finite-dimensional
representations of the alternating group $\mathbb A_5$ and the
pointed fusion categories $\C(\mathbb Z_p, \omega)$ of
finite-dimensional $\mathbb Z_p$-graded vector spaces, where $p$
is a prime number, with associativity constraint determined by a
$3$-cocycle $\omega \in H^3(\mathbb Z_p, k^*)$.

%

\medbreak The paper is organized as follows. In Section
\ref{prels} we recall some basic definitions and results on fusion
categories. Some useful lemmas are also contained in this section.
In Section \ref{sec2} we consider integral fusion categories and
state several results on the  possible simple Frobenius-Perron
dimensions that will be used later. Most of them have appeared in
the literature in the context of representations of semisimple
Hopf algebras. In Section \ref{sec3} we prove our main results on fusion categories of small dimension; the computations in this section are partly handled by a computer.

\section{Preliminaries}\label{prels}

Throughout this paper we shall work over an algebraically closed base field $k$ of characteristic zero. Recall that a fusion category over $k$ is a $k$-linear semisimple
rigid tensor category $\C$ with finitely many isomorphism classes
of simple objects, finite-dimensional hom spaces, and such that
the unit object $\textbf{1}$ of $\C$ is simple. We refer the
reader to \cite{ENO} for the main notions about fusion categories
used throughout.

\medbreak Let $\C$ be a fusion category over $k$. Let also $\Irr
(\C)$ and $\Pic(\C)$ denote the set of isomorphism classes of
simple and invertible objects of $\C$, respectively. Then $\Irr
(\C)$ is a basis of the Grothendieck ring $K_0(\C)$ of $\C$ and
$\Pic(\C)$ is a subgroup of the group of units of $K_0(\C)$.

 The \emph{Frobenius-Perron dimension} of $x\in \Irr
(\C)$ is the Frobenius-Perron eigenvalue of the matrix of left
multiplication by $x$ in the Grothendieck ring of $\C$. Thus
$\FPdim$ extends to a ring homomorphism $\FPdim : K_0(\C) \to
\mathbb R$. This is the unique ring homomorphism  that takes
positive values in all elements of $\Irr (\C)$. The
Frobenius-Perron dimension of $\C$ is the number $\FPdim \C =
\sum_{x \in \Irr (\C)} (\FPdim x)^2$.

If $X$ is an object of $\C$, $\FPdim X$ is the Frobenius-Perron
dimension of the class of $X$ in $K_0(\C)$. We have $\FPdim X \geq
1$, for all objects $X$ of $\C$. Moreover $\FPdim X = 1$ if and only
if $X$ is an invertible object.

Let $y \in K_0(\C)$ and write $y = \sum_{x \in \Irr
(\C)} m(x, y) x$, where $m(x, y) \in \mathbb Z$. The  integer
$m(x, y)$ is called the \emph{multiplicity} of $x$ in $y$. This
extends to a bilinear form $m: K_0(\C) \times K_0(\C) \to \mathbb
Z$. If $x$ and $y$ represent the class of the objects $X$ and $Y$
of $\C$, respectively, then we have $m(x, y)=\dim \Hom_\C(X, Y)$.

Let $x, y, z \in K_0(\C)$. Then we have
$m(x,y)=m(x^*,y^*)$, and
$$m(x,yz)=m(y^*,zx^*)=m(y,xz^*).$$
Let $x, y \in \Irr(\C)$. Then for each $g \in \Pic(\C)$ we have
$m(g,xy)=1$ if and only if $y=x^*g$ and $0$ otherwise. In
particular, $m(g,xy)=0$ if $\FPdim x\neq \FPdim y$. Let $x\in
\Irr(\C)$. Then for all $g \in \Pic(\C)$, $m(g,xx^{*})>0$ if and
only if $m(g,xx^{*})= 1$ if and only if $gx=x$. The set of
isomorphism classes of such invertible objects will be denoted
$G[x]$. Thus $G[x]$ is a subgroup of $\Pic(\C)$ of order at most
$(\FPdim x)^2$. In particular, for all $x\in \Irr(\C)$, we have a
relation
$$xx^*=\sum_{g\in G[x]}g+\sum_{y\in \Irr(\C),\FPdim y >1} m(y, xx^*) y.$$

In fact, the group $\Pic(\C)$ acts on the set $\Irr(\C)$ by left
multiplication. This action preserves Frobenius-Perron dimensions
and, for $x\in \Irr(\C)$, $G[x]$ is the stabilizer of $x$ in
$\Pic(\C)$.

For every $\alpha\in \mathbb R_+$, we shall use the notation
$\Irr_\alpha(\C)$ to indicate the set of
isomorphism classes of simple objects of $\C$ of Frobenius-Perron
dimension $\alpha$. So that $\Pic(\C) = \Irr_1(\C)$.
Notice that, if $|\Irr_\alpha(\C)|=1$ for some $\alpha
\in \mathbb R_+$, and $x\in \Irr_\alpha(\C)$, then $G[x]=\Pic(\C)$.

\medbreak A fusion subcategory of $\C$ is a full tensor
subcategory $\D$ such that if $X$ is an object of $\C$ isomorphic
to a direct summand of an object $Y$ of $\D$, then $X$ is in $\D$.
If $\D$ is a fusion subcategory of $\C$, then $\D$ is in fact a
fusion category and $\FPdim \D$ divides $\FPdim \C$, that is, the
quotient $\FPdim \C/\FPdim \D$ is an algebraic integer.

Fusion subcategories of $\C$ correspond to \emph{fusion
subrings} of the Grothendieck ring of $\C$, that is, subrings $R$
with the property that for all  $a \in R$ and for all $x \in \Irr
(\C)$ such that $m(x, a) \neq 0$, we have $x \in R$. In other
words, $R$ is a subring which is spanned by a subset of
$\Irr(\C)$. A subset $X$ of $\Irr(\C)$ spans a fusion subring of
$K_0(\C)$ if and only if the product of elements of $X$ decomposes
as a sum of elements of $X$.

The group $\Pic(\C)$ of invertible objects of $\C$
generates a fusion subcategory $\C_{\pt}$ of $\C$, which is the
unique largest pointed fusion subcategory of $\C$. In particular,
the order of $\Pic(\C)$ coincides with $\FPdim \C_{\pt}$ and
therefore it divides $\FPdim \C$.

\medbreak Let $G$ be a finite group. A fusion category $\C$ is called a $G$-extension of a fusion category $\D$ if it admits a faithful grading $\C = \oplus_{g \in G} \C_g$ by the group $G$, satisfying that the tensor product of $\C$ maps $\C_g \times \C_h \to \C_{gh}$ and $(\C_g)^* = \C_{g^{-1}}$, and such that the trivial homogeneous component $\C_e$ is equivalent to $\D$.

Recall from \cite{gel-nik} that a fusion category $\C$ admits a canonical faithful grading $\C = \oplus_{g\in U(\C)}\C_g$, whose trivial component $\C_e$ coincides with the \emph{adjoint fusion subcategory} $\C_{\ad}$. The group  $U(\C)$ is called the \emph{universal grading group} of $\C$. Any faithful grading $\C = \oplus_{g \in G} \C_g$ by a group $G$ comes from a group epimorphism $U(\C) \to G$.

\begin{lem}\label{ker-U} Let $\C$ be a fusion category and let $\Z(\C)$ be its Drinfeld center. Consider the group homomorphism  $F_0: \Pic(\Z(\C)) \to \Pic(\C)$ induced by the  forgetful functor  $F: \Z(\C) \to \C$. Then the following hold:
\begin{itemize}\item[(i)] $\C$ is faithfully graded by the group $\widehat N$, where $N$ is the kernel of $F_0$.

\item[(ii)]  Suppose $U(\C) = 1$. Then the group homomorphism $F_0$ is injective. \end{itemize} \end{lem}

\begin{proof} (i)  Observe that  $F$ induces by restriction a tensor functor $F: \langle N \rangle \to \vect$, where $\langle N \rangle$ is the (pointed) fusion subcategory of $\Z(\C)$ generated by $N$. Therefore $\langle N \rangle$ is equivalent as a fusion category to the category of representations of the commutative Hopf algebra $k^N$. Thus part (i) follows from \cite[Proposition 5.10]{ENO}.

(ii) The assumption implies that $\C$ admits no faithful group grading.   Therefore (ii) follows from (i) (observe that the group of invertible objects of $\Z(\C)$ and thus also its subgroup $N$ are abelian, whence $\widehat N \cong N$).
\end{proof}

\medbreak The fusion category $\C$ is called \emph{ weakly
integral} if $\FPdim \C$ is a natural number. If $\FPdim X \in
\mathbb Z$, for all object $X$ of $\C$, $\C$ is called
\emph{integral}. Suppose $\C$ is a weakly integral fusion
category. It follows from \cite[Theorem 3.10]{gel-nik} that either
$\C$ is integral, or $\C$ is a $\mathbb Z_2$-extension of a fusion
subcategory $\D$. In particular, if $\FPdim \C$ is odd or if $\C =
\C_{\ad}$, then $\C$ is necessarily integral.

\begin{lem}\label{lem1}
Let $x\in \Irr(\C)$. Then the following hold:

(i)\,The order of $G[x]$ divides $(\FPdim x)^2$.

(ii)\,The order of $\Pic(\C)$ divides $n(\FPdim x)^2$, where $n$
is the number of non-isomor\-phic simple objects of
Frobenius-Perron dimension $\FPdim x$.
\end{lem}

\begin{proof} It is enough to show (i). Consider the fusion subcategory $\D \subseteq \C_{\pt}$ generated by $G[x]$. Regard $\C$ as module category over $\D$
with respect to the action given by the tensor product of $\C$. If
$X$ is a simple object of $\C$ representing $x$, then the
definition of $G[x]$ implies that $g \otimes X \cong X$, for all
simple object $g$ of $\D$. Therefore, the full abelian subcategory
$\mathcal M$ whose objects are isomorphic to direct sums of copies
of $X$ is an indecomposable $\D$-module subcategory of $\C$. Since
$X$ is the unique simple object of $\mathcal M$ up to isomorphism,
then $\FPdim \mathcal M = (\FPdim x)^2$. Part (i) now follows from
\cite[Proposition 8.15 and Remark 8.17]{ENO}.
\end{proof}

\begin{cor}\label{pdivn} Suppose that $\Pic(\C)$ is of prime order $p$.
Assume in addition that $p > \vert \Irr_\alpha(\C)\vert > 0$, for
some $\alpha \neq 1$. Then $G[x] = \Pic(\C)$, for all $x \in
\Irr_\alpha(\C)$. In particular, $p$ divides $\alpha^2$.
\end{cor}

\begin{proof} Let $x \in \Irr_\alpha(\C)$. Decomposing the set $\Irr_\alpha(\C)$ into disjoint
orbits under the action of $\Pic(\C)$, we get that $G[x] \neq \1$.
Then necessarily $G[x] = \Pic(\C)$. It follows from Lemma
\ref{lem1} that $p$ divides $\alpha^2$.
\end{proof}

\begin{rem} In the notation of the proof of Lemma \ref{lem1}, the rank-one module category $\mathcal M$ corresponds to a fiber functor over $\D$. Being a pointed fusion category, this implies that  $\D$ is equivalent as a fusion category to the category $\C({G[x]}) = \C({G[x]}, 1)$ of finite-dimensional $G[x]$-graded vector spaces.

The proof of Corollary \ref{pdivn} shows moreover that, under the
assumptions of the lemma, we have an equivalence of fusion
categories $\C_{\pt} \cong \C(\mathbb Z_p)$.
\end{rem}

\medbreak Let $1=d_0, d_1,\cdots, d_s$, $s \geq 0$, be positive
real numbers such that $1 = d_0 < d_1< \cdots < d_s$, and let
$n_1,n_2,\cdots,n_s$ be positive integers. We shall
say that $\C$ is \emph{of type} $(d_0,n_0;
d_1,n_1;\cdots;d_s,n_s)$ if, for all $i = 0, \cdots, s$, $n_i$ is
the number of the non-isomorphic simple objects of
Frobenius-Perron dimension $d_i$.

Hence, if $\C$ is of type $(d_0,n_0; d_1,n_1;\cdots;d_s,n_s)$,
then $n_0$ equals the order of $\Pic(\C)$ and we have a relation
\begin{equation}\label{gral} \FPdim \C = n_0 + d_1^2n_1 + \dots +
d_s^2n_s. \end{equation}

For each $i = 1, \dots, s$, let $x^{l}_i \in \Irr(\C)$, $1 \leq l
\leq n_i$, such that $\FPdim x^l_i = d_i$. Decomposing the tensor
product $x^l_i\otimes {x^l_i}^*$ into a direct sum of simple
objects, and comparing dimensions,  we get a relation $$d_i^2 =
\vert G[x^l_i] \vert + \sum_{j = 1}^s m_jd_j,$$ where $m_j = \sum_t
m(x^t_j, x^l_i\otimes {x^l_i}^*)$, for all $j$. Hence $m_j \geq
0$, and in addition $m_j = \sum_t m(x^t_j, x^l_i\otimes {x^l_i}^*)
= \sum_t m(x^l_i, x^t_j\otimes x^l_i) \leq \sum_t \FPdim x^t_j =
n_jd_j$.

\begin{lem}\label{prod-simples-categorico}
Let $x, x' \in \Irr(\C)$. Then the following are equivalent:
\begin{enumerate} \item[(i)]$x^*x' \in \Irr(\C)$.

\item[(ii)] For all $\1 \neq y \in \Irr(\C)$, either $m(y, xx^*)
= 0$ or $m(y, x'{x'}^*) = 0$. \end{enumerate}
\end{lem}

\begin{proof} See \cite[Lemma 6.1]{Bichon}. \end{proof}

We state the following lemma for future use. It is a consequence of the solvability of fusion categories of Frobenius-Perron dimension $p^aq^b$ \cite[Theorem 1.6]{ENO2} .

\begin{lem}\label{pa-qb} Suppose $\FPdim \C = p^aq^b$, where $p$ and $q$ are prime numbers, $a, b \geq 0$,
such that $a+b \neq 0$. Then the group $\Pic(\C)$ is not trivial.
\end{lem}

\begin{proof} The proof is by induction on $\FPdim \C$.
The assumption implies that $\C$ is solvable \cite[Theorem
1.6]{ENO2}. Then $\C$ is either a $G$-equivariantization or a
$G$-extension of a fusion category $\D$,  where $G$ is a cyclic
group of prime order. If $\C$ is a $G$-equivariantization, then
$\C$ contains a fusion subcategory equivalent to $\rep G$, hence
the lemma follows in this case. Otherwise, $\C$ has a faithful
grading $\C = \oplus_{g \in G}\C_g$, where $\C_e = \D$. Since
$\FPdim \D$ divides $\FPdim \C$, then either $\FPdim \D = 1$ and
$\C$ is pointed, or $1 < \FPdim \D = p^{a'}q^{b'}$, with $|G|
\FPdim \D = \FPdim \C$. In the first case we are done. In the last
case, $\D$  has nontrivial invertible objects, by induction. Hence
so does $\C$. This finishes the proof of the lemma. \end{proof}

\section{Integral fusion categories}\label{sec2}

Along this section $\C$ will denote an integral fusion category
over $k$. We collect here some useful facts about the structure of
$\C$ that will be useful in the proof of our main results. Some of
these results have been established in the literature in the case
where $\C$ is the category of finite-dimensional representations
of a semisimple Hopf algebra, but their proofs only make use of
the properties of the Grothendieck ring explained in Section
\ref{prels}. Therefore these proofs also work  \textit{mutatis
mutandis} in the fusion category setting, and thus they are
omitted in our exposition.

\medbreak In what follows we assume that $\C$ is of type $(1, n_0;
d_1, n_1; \cdots; d_s, n_s)$.

\begin{lem}\label{lem3} Let $d = \operatorname{gcd}(d_1, \cdots, d_s)$. Then
$d$ divides $n_0$.
\end{lem}

In particular, if $n_0$ is a prime number and $d \neq 1$, then $d
= p$. On the other hand, if $\C$ is of type $(1,n_0; d,n)$, then
$d$ divides $n_0$ (see \cite[Lemma 5.2]{dong2}, \cite[Theorem 5.1 (b)]{LR}), hence in this
case $\C$ is of Frobenius type.

\begin{proof}
Let $x \in \Irr(\C)$ such that $\FPdim x > 1$. Decomposing the
product $xx^*$ and taking Frobenius-Perron dimensions, we get that
$d$ divides the order of $G[x]$. This implies the lemma.
\end{proof}

\begin{lem}\label{prod-leq2} Let $\C$ be an integral fusion category. Suppose that one of the following conditions hold:

(a)\, $\Irr_4(\C) = \emptyset$ and $G[x] \neq \1$, for all $x \in \Irr_2(\C)$, or

(b) \, $G[x] \cap G[x'] \neq \1$, for all $x, x' \in \Irr_2(\C)$.

Then $\C$ has a fusion subcategory of type $(1, n_0; 2, n)$, where $n = \vert
\Irr_2(\C)\vert$.
\end{lem}

\begin{proof} Assume first that (a) holds. Let $x, y \in \Irr_2(\C)$. Note that the invertible objects appearing
in $xy$ with positive multiplicity form a coset of the stabilizer $G[x]$ in
$\Pic(\C)$, and therefore there is an even number of them. The remaining simple constituents (if any) of $xy$
should therefore belong to $\Irr_2(\C)$.

In view of Lemma \ref{prod-simples-categorico},
assumption (b) also implies that the product of any two elements $x, y
\in \Irr(\C)$ such that $\FPdim x, \FPdim y \leq 2$ decomposes a
sum of elements of $\Irr(\C)$ of Frobenius-Perron dimension $\leq
2$. See \cite[Theorem 2.4.2]{Natale4}.

Then we have shown that both assumptions (a) and (b) imply that the set $\{ x \in
\Irr(\C)\vert \, \FPdim x \leq 2 \}$ spans a fusion subring of
$K_0(\C)$, corresponding to a fusion subcategory of the prescribed
type.
\end{proof}

\begin{lem}\label{lema-8} Suppose that $\Irr_2(\C)$ is odd and the order of $\Pic(\C)$ is divisible by $4$.
Then $\C$ has a non-pointed fusion subcategory of Frobenius-Perron
dimension $8$. \end{lem}

\begin{proof} It follows from \cite[Proposition 2.1.3]{Natale4}.
\end{proof}

The following theorem is a restatement of \cite[Theorem
11]{Nichols} in the context of fusion categories. The theorem has
found many applications in the classification of low-dimensional
semisimple Hopf algebras, since such semisimple Hopf algebras
often have irreducible characters of degree $2$.

\begin{thm}\label{thm1}
Suppose $x \in \Irr(\C)$ is such that  $\FPdim x = 2$. Then at least one of the following holds:

(i)\, $G[x] \neq \1$.

(ii)\, $\C$ has a fusion subcategory $\D$ of type $(1, 2; 2, 1; 3, 2)$, such that $x \notin \Irr(\D)$ which has an invertible object $g$
of order $2$ such that $gx\neq x$.

(iii)\, $\C$ has a fusion subcategory of type $(1, 3; 3, 1)$ or $(1, 1; 3, 2; 4, 1; 5, 1)$. \qed
\end{thm}

It follows that if $G[x] = \1$, then $\FPdim \C$ is divisible by $12$, $24$, or $60$. In particular,
$\FPdim \C$ is always even.

\medbreak In the next lemmas we list some consequences of Theorem
\ref{thm1}.

\begin{lem}\label{lem4} (i)\, Suppose that the order of $\Pic(\C)$ is odd. Assume that $\Irr_2(\C)\neq \emptyset$.
Then $\Irr_3(\C)\neq \emptyset$. Further, if $\Irr_4(\C)= \emptyset$ or $\Irr_5(\C)= \emptyset$,
then $\C$ has a fusion subcategory of type $(1, 3; 3, 1)$ and
hence $12$ divides $\FPdim \C$.

(ii)\, Suppose that $\Pic(\C) = \1$ and $\Irr_2(\C)\neq
\emptyset$. Then $\Irr_3(\C)$, $\Irr_4(\C)$ and $\Irr_5(\C)$ are
non-empty and $60$ divides $\FPdim \C$.

(iii)\, Suppose that $\Pic(\C)$ is of prime order $p \neq 3$ and
$\Irr_2(\C)\neq \emptyset$. If $\Irr_4(\C)= \emptyset$, then $\Pic(\C)$
is of order $2$ and $\C$ has a fusion subcategory of
Frobenius-Perron dimension $2 + 4\vert \Irr_2(\C) \vert$.

(iv)\, Suppose that $\Pic(\C)$ is of prime order $p$. Assume that
$\Irr_2(\C)\neq \emptyset$ and $\FPdim \C$ is not divisible by $12$.
Then $p = 2$ and $\C$ has a fusion subcategory of Frobenius-Perron
dimension $2 + 4\vert \Irr_2(\C) \vert$.
\end{lem}

\begin{proof} (i) The assumption implies that $G[x] = \1$, for all
$x \in \Irr_2(\C)$. Moreover, $\C$ cannot have fusion
subcategories of type $(1, 2; 2, 1; 3, 2)$ or $(1, 3;
3, 1)$. Part (i) then follows from Theorem \ref{thm1}.

(ii) Follows directly from Theorem \ref{thm1}.

(iii) Let $x \in \Irr_2(\C)$. Assume first that $G[x] = \1$. By
\cite[Remark 2.2.2 (i)]{Natale4}, there exist $y \in \Irr_3(\C)$ such
that $|G[y]| = 3$. Then $3$ divides $|\Pic(\C)|$, which is a
contradiction. Therefore $G[x] \neq \1$ and hence $G[x] =
\Pic(\C)$ is of order $2$. Since this holds for all $x \in
\Irr_2(\C)$, part (iii) follows from  Lemma \ref{prod-leq2}.

(iv)\, Let $x \in \Irr_2(\C)$. Since $12$ does not divide $\FPdim
\C$, Theorem \ref{thm1} implies that $G[x] \neq \1$. Hence $G[x] =
\Pic(\C)$, for all $x \in \Irr_2(\C)$, and therefore $p = 2$. Part
(iv) follows from Lemma \ref{prod-leq2}.
\end{proof}

\begin{lem}\label{lem2}
Suppose that $\FPdim x = 2$ for some  $x \in \Irr(\C)$. Let $n_0$
be the order of $\Pic(\C)$ and let $s = |\Irr_2(\C)|$. Assume one
of the following conditions hold:

(i)\, $12$ does not divide $\FPdim \C$ and $\Irr_4(\C) =
\emptyset$.

(ii)\, $n_0 = 2$ and  $\FPdim \C$ is not divisible by $24$ or
$60$.

(iii)\, $n_0 = 2$ and $\Irr_3(\C) = \emptyset$.

(iv)\, $\Irr_3(\C) = \Irr_4(\C) = \emptyset$.

Then $\C$ has a fusion subcategory of type $(1, n_0; 2, s)$. In
particular, $\FPdim \C$ is divisible by $n_0+4s$.
\end{lem}

\begin{proof}
(i)\,Suppose that $12$ does not divide $\FPdim \C$. Then Theorem
\ref{thm1} shows that $G[x]\neq \1$ for all $x\in \Irr_2(\C)$.
Hence (i) follows from Lemma \ref{prod-leq2}.
%

(ii)\, In this case $\C$ cannot have fusion subcategories of type
$(1, 3; 3, 1)$. It follows from Theorem \ref{thm1} that
$G[x]=\Pic(\C)$, for all $x\in \Irr_2(\C)$. Then (ii) follows from
Lemma \ref{prod-leq2}.

(iii)\, By Theorem \ref{thm1}, $G[x] \neq \1$, for all $x \in
\Irr_2(\C)$.
Hence (iii) follows from Lemma \ref{prod-leq2}.

(iv)\, Since $\Irr_3(\C) = \emptyset$, Theorem \ref{thm1} implies
that $G[x] \neq \1$, for all $x \in \Irr_2(\C)$. Then $G[x] =
\Pic(\C)$, for all $x \in \Irr_2(\C)$, and (iv) follows also from
Lemma \ref{prod-leq2}.
\end{proof}

The next proposition is a restatement of some results of this
section as well as other results in the literature, in terms of
fusion category types. It provides us with conditions that can be
easily handled by a computer.

\begin{prop}\label{prop1}
Let $\C$ be an integral fusion category over $k$.

(i)\, If $60$ does not divide $\FPdim \C$, then $\C$ cannot be of
type $(1,1; 2,m; \cdots)$.

(ii)\, If $8$ does not divide $\FPdim \C$, then $\C$ cannot be of
type $(1,n_0;2,m; \cdots)$, where $n_0$ is divisible by $4$ and $m$
is odd.

(iii)\, If $n_0$ does not divide $\FPdim \C$ or $n_id_i^2$, for some
$2\leq i\leq s$, then $\C$ cannot be of type
$(1,n_0;d_2,n_2;\cdots;d_s,n_s)$.

(iv)\, If $t$ does not divide $n_0$, then $\C$ cannot be of type
$(1,n_0; t,n)$.

(v)\, If $s \leq 2$,  then $\C$ cannot be of type $(1,1;
d_1,n_1;\cdots,d_s,n_s)$.

(vi)\, If neither $24$ nor $60$ nor
$2+4m$ divide $\FPdim \C$, then $\C$ cannot be of type
$(1,2;2,m;\cdots)$.

(vii)\, If neither $12$ nor $n_0 + 4m$ divide $\FPdim \C$, and
$\C$ does not have simple objects of Frobenius-Perron dimension
$4$, then $\C$ cannot be of type $(1,n_0;2,m;\cdots)$.

(viii)\, If $n_0 + 4m$ does not divide $\FPdim \C$ and $d\geq 5$,
then $\C$ cannot be of type $(1,n_0; 2,m; d,n; \cdots)$.
\end{prop}

\begin{proof}
Part (i) follows from Theorem \ref{thm1}. Part (ii) follows from
Lemma \ref{lema-8}. Part (iii) follows from Lemma \ref{lem1}. Part
(iv) follows from Lemma \ref{lem3}. Part (v) follows from
\cite[Lemma 11]{szhu}. Parts (vi), (vii) and (viii) follow from
Lemma \ref{lem2}.
\end{proof}

We end this section with some further applications of Theorem
\ref{thm1} that will be used later on (c.f. the proof of Theorem
\ref{thm3}).

\begin{prop}\label{irr2-zdec} Let $n \geq 1$ be an odd natural number and let $\C$ be an integral fusion category such that
$\FPdim \C = 2n$. Assume in addition that $\Irr_2(\Z(\C)) \neq
\emptyset$. Then $\C$ is a $\mathbb Z_p$-extension or a $\mathbb Z_p$-equivariantization of a fusion category $\D$, for some prime number $p$.
\end{prop}

\begin{proof} It will be enough to show that $\Z(\C)$ has a nontrivial Tannakian
subcategory $\E$. In this case, $\E \cong \rep \Gamma$
where the order of $\Gamma$ divides $\FPdim \C = 2n$: indeed,  $\E
\subseteq \E'$, where $\E'$ denotes the M\"uger's centralizer of $\E$ in $\Z(\C)$ and, by \cite[Theorem 3.2]{mueger}, $\FPdim \E \FPdim \E' = (\FPdim \C)^2$. In particular, $\Gamma$ is solvable (see e.g. \cite[Theorem 1.35]{isaacs}), and it follows that $\E$, and
hence also $\Z(\C)$, contains a Tannakian subcategory of prime
dimension. By \cite[Propositions 2.9 and 2.10]{ENO2}, $\C$ is an
equivariantization or an extension of a fusion category $\D$.

Consider the forgetful functor $F: \Z(\C) \to \C$. Suppose first that the universal grading group $U(\C)$ is not trivial. Since $\C$ is a $U(\C)$-extension of its adjoint fusion subcategory $\C_{ad}$, it follows from \cite[Proposition 2.9 (ii)]{ENO2} that $\Z(\C)$ contains the category $\rep U(\C)$ as a Tannakian subcategory. The above discussion shows that the proposition holds in this case. Therefore we may assume that $U(\C)$ is trivial. It follows from Lemma \ref{ker-U} (ii) that $F$ induces an injective group homomorphism $\Pic(\Z(\C)) \to \Pic(\C)$.
Then the order of $\Pic(\Z(\C))$
cannot be divisible by $4$ (since $4$ does not divide $\FPdim
\C$).
By \cite[Theorem 2.11]{ENO2}, the Frobenius-Perron dimensions of simple objects of $\Z(\C)$
divide $\FPdim \C$. Hence
$\Z(\C)$ has no simple objects of Frobenius-Perron dimension $4$.

Suppose that $\Z(\C)$ has a fusion subcategory $\D$ of
type $(1, 3; 3, 1)$. Let $\Z_2(\D) \subseteq \D$ denote the M\"
uger center of $\D$. If $\Z_2(\D)$ is trivial, then $\D$ is
non-degenerate and in view of \cite[Theorem 4.2]{mueger}, \cite[Theorem 3.13]{DGNOI}, there is an equivalence of braided fusion
categories $\Z(\C) \cong \D \boxtimes \D'$, where $\D'$ is the M\"
uger centralizer of $\D$. In particular, $\D'$ must have simple
objects of Frobenius-Perron dimension $2$. But this is impossible
since $\FPdim \D'$ is odd \cite[Theorem 2.11]{ENO2}.
Hence $\Z_2(\D)$ is not trivial, and therefore $\rep \mathbb Z_3
\subseteq \Z_2(\D)$. But $\rep \mathbb Z_3$, being an
odd-dimensional symmetric fusion category, is Tannakian \cite[Corollary 2.50 (i)]{DGNOI}. Hence we
are done in this case.  We may thus assume that $\Z(\C)$ has no
fusion subcategory $\D$ of type $(1, 3; 3, 1)$.

\medbreak  Let $x \in \Irr_2(\Z(\C))$. Since $\FPdim \Z(\C)$ is
not divisible by $24$ and $\Irr_4(\Z(\C)) = \emptyset$, it
follows from Theorem \ref{thm1} and the discussion in the previous
paragraph, that $G[x] \neq \1$. Since the
order of $\Pic(\Z(\C))$ is not divisible by $4$, then $G[x]$ is of
order $2$, for all $x \in \Irr_2(\Z(\C))$. Furthermore, the
abelian group $\Pic(\Z(\C))$ has a unique subgroup of order $2$,
and therefore $G[x] = G[x'] \simeq \mathbb Z_2$, for all $x, x'
\in \Irr_2(\Z(\C))$. Lemma \ref{prod-leq2} implies that $\Z(\C)$
has a fusion subcategory $\D$ of type $(1, n_0; 2, m)$, where $n_0
= \vert \Pic(\C) \vert$ and $m \geq 1$. Since  $\D$ is braided,
then $\D$ contains a Tannakian subcategory, namely, the M\" uger
center of the fusion subcategory generated by a self-dual object
of $\Irr_2(\D)$ (see \cite[Lemma 4.7]{cd2-wint}). This finishes
the proof of the proposition.
\end{proof}

\begin{cor}\label{90-6} Let $\C$ be an integral fusion category such that
$\FPdim \C = 90$. Suppose that $\C$ has a fusion subcategory $\D$
of Frobenius-Perron dimension $6$. Then $\C$ is weakly
group-theoretical. \end{cor}

\begin{proof} By Proposition \ref{irr2-zdec}, we may
assume that $\Irr_2(\Z(\C)) = \emptyset$. Indeed, if $\Z(\C)$ had a simple object of Frobenius-Perron dimension $2$, then  $\C$ would be an
equivariantization or an extension of a fusion category $\D$, whose Frobenius-Perron dimension divides $90$ and is less than $90$.
In view of the results of \cite{ENO2}, the fusion category $\D$
(and then also $\C$) must be weakly group-theoretical.
As in the proof of Proposition \ref{irr2-zdec}, it will be enough to show that
$\Z(\C)$ has a nontrivial Tannakian subcategory. To do this, we
shall follow the lines of the proof of \cite[Theorem 9.16]{ENO2}.

\medbreak Let $I: \C \to \Z(\C)$ be the left adjoint of  the forgetful
functor $F: \Z(\C) \to \C$.  Then $A = I(\1)$ is a commutative
algebra in $\Z(\C)$ such that $\C \cong \Z(\C)_A$ as fusion
categories \cite{ENO2}. We may assume that $A$ contains no
nontrivial invertible object of $\Z(\C)$: indeed, if $m(g, A) \neq
0$, where $\1 \neq g \in \Pic(\Z(\C))$, then $m(\1, F(g)) \neq 0$
and therefore $F(g) \cong \1$. By Lemma \ref{ker-U}, $\C$ is faithfully graded by $\widehat N$, where $N = \langle g \rangle$ is the subgroup
generated by $g$ (in particular, $\Z(\C)$ contains a Tannakian subcategory) and we are done.

The fusion subcategory $\D \subseteq \C$ corresponds to a subalgebra $B$ of $A$
such that $\FPdim B = \FPdim \C/\FPdim \D = 15$. In view of
\cite[Theorem 2.11]{ENO2}, the Frobenius-Perron dimensions of
simple objects of $\Z(\C)$ divide  $90$.  Hence the possible decompositions of $B$ as an
object of $\Z(\C)$ are the following:
\begin{equation}\label{decomposition-b}\1 \oplus X_3 \oplus X'_3\oplus X''_3 \oplus X_5,
\quad \1 \oplus X_3 \oplus X_5 \oplus X_6, \quad \1 \oplus X_5
\oplus X_9,\end{equation} where $X_3$, $X'_3$, $X''_3$ are simple objects of
Frobenius-Perron dimension $3$, and $X_5$, $X_6$, $X_9$ are simple
objects of Frobenius-Perron dimensions $5$, $6$ and $9$,
respectively.

It follows that $\Z(\C)$ has simple objects of prime power
dimension, and therefore it contains a nontrivial symmetric
subcategory \cite[Corollary 7.2]{ENO2}. We may assume $\Z(\C)$ has
a unique nontrivial symmetric subcategory $\E$, which is
equivalent to the category of super vector spaces. Note that $\E
\subseteq \Z(\C)_{\pt}$ and $\FPdim(\E) = 2$. Let $\Z = \E'
\subseteq \Z(\C)$ be the M\" uger  centralizer of $\E$. Then $\E
\subseteq \Z$, since $\E$ is symmetric, and $\Z$ is a slightly
degenerate (as $\Z_2(\E') = \E$) fusion subcategory of
Frobenius-Perron dimension $(\FPdim \C)^2/2$.

In particular the group $G =\Pic(\Z(\C)) \neq \1$ and, since
$\Pic(\Z(\C))$ is isomorphic to the universal grading group of
$\Z(\C)$ \cite[Theorem 6.2]{gel-nik}, there is a faithful
$G$-grading on $\Z(\C)$, with trivial component  $\Z(\C)_{\ad}$.
By \cite[Corollary 6.8]{gel-nik}, we have $\Z(\C)'_{\ad} =
\Z(\C)_{\pt}$. Since $\E \subseteq \Z(\C)_{\pt}$, then
$\Z(\C)'_{\pt} = \Z(\C)_{\ad} \subseteq \Z$.

Note that if the slightly degenerate integral braided category
$\Z$ has a simple object of odd prime power dimension, then it
follows from \cite[Proposition 7.4]{ENO2} that $\Z$, and thus also
$\Z(\C)$, contains a nontrivial Tannakian subcategory. Then  we
are done in this case.

We may therefore assume that $\Z$ has no simple object of odd prime power dimension. Consider first the case where $\Irr_3(\Z(\C)) \neq
\emptyset$ and let $X$ be a simple object with $\FPdim X = 3$. We
have $X\otimes X^* \in \Z(\C)_{\ad}$. Since $\Z(\C)_{\ad} \subseteq \Z$, then $X\otimes X^*$ has no simple constituents of odd
prime power dimension. This implies that $X\otimes X^* =
\bigoplus_{a \in G[X]} a \bigoplus Y$, where  $\vert G[X] \vert =
3$ and $Y$ is a simple object of Frobenius-Perron dimension $6$, or $X \otimes X^* = \bigoplus_{a \in G[X]} a$ with $|G[X]| = 9$.

Since $G[X] \subseteq \Pic(\Z(\C)) \cap \Z(\C)_{\ad} \subseteq \Z(\C)_{\ad}' \cap \Z(\C)_{\ad}$, then 
the fusion subcategory $\langle G[X] \rangle$ generated by $G[X]$ is symmetric and therefore Tannakian, because it is odd dimensional \cite[Corollary 2.50 (i)]{DGNOI}.

We may therefore assume that $\Irr_3(\Z(\C)) = \emptyset$. Then
the Frobenius-Perron dimensions of simple objects of $\Z(\C)$ are
among the numbers $1, 5, 6, 9, 10, 15, 18, 30$ and $45$. Moreover,
in view of the possible decompositions in \eqref{decomposition-b},
$\Irr_5(\Z(\C))$ and $\Irr_9(\Z(\C))$ are both non-empty.

\medbreak Let $X \in \Z(\C)$ be a simple object of
Frobenius-Perron dimension $5$ and let $\mathcal B$ denote the
fusion subcategory generated by $X$. We claim that $\mathcal B'$
is a nontrivial proper fusion subcategory  of $\Z(\C)$. Indeed,
there must exist a simple object $Y$ of Frobenius-Perron dimension
not divisible by $5$ such that $X$ and $Y$ projectively centralize
each other. Otherwise $S_{X, Y} = 0$ for all such simple objects
\cite[Lemma 7.1]{ENO2}. As in the proof of \cite[Corollary
7.2]{ENO2}, the orthogonality of columns of the $S$-matrix imply a
relation
\begin{equation*}0 = \sum_{Y \in \Irr(\Z(\C))}
\frac{S_{X, Y}}{5} \, \FPdim Y = 1 + \sum_{\substack{\1 \neq Y\\
(\FPdim Y, 5)\neq 1}} \frac{S_{X, Y}}{5} \, \FPdim
Y,\end{equation*} which is impossible. Thus there exists a simple
object $Y$ with $\FPdim Y = 6, 9$ or $18$ and such that $X$ and
$Y$ projectively centralize each other. Then $Y\otimes Y^* \in
\mathcal B'$ and the claim follows (note that $\mathcal B' \neq
\Z(\C)$ since otherwise $\vect = \Z(\C)' = \mathcal B'' = \mathcal
B$).

Let $\D = \mathcal B'$. If $\D$ is degenerate, then its M\" uger
center $\Z_2(\D)$ is a symmetric fusion subcategory of $\Z(\C)$
and therefore $\Z_2(\D) = \mathcal E \subseteq \D$. But this
implies that $\mathcal B = \D' \subseteq \mathcal E'= \Z$, thus
$X$ is a simple object of odd prime power dimension of $\Z$ and we
are done.

We may therefore assume that $\D$ is non-degenerate. Hence $\Z(\C)
= \D \boxtimes \D' = \D \boxtimes \mathcal B$, and $\mathcal B$ is
non-degenerate. Since $\mathcal B$ has a simple object of prime
dimension $5$, it follows from \cite[Corollary 7.2]{ENO2} that
$\mathcal B$ contains a symmetric fusion subcategory. Hence
$\mathcal E \subseteq \mathcal B$. In particular, $\mathcal E
\nsubseteq \D$ and we obtain that $\D$ has no simple objects of
prime power dimension. This implies that all simple objects of
Frobenius-Perron dimension $5$ and $9$ belong to $\mathcal B$.

Let $Z$ be a non-invertible simple constituent of $Y \otimes Y^*$,
where $Y$ is a simple object with $\FPdim Y = 6, 9$ or $18$ that
projectively centralizes $X$. Then $Z \in \D$ and the
Frobenius-Perron dimension of $Z$ is either $6, 10, 15, 18, 30$ or
$45$. On the other hand, if $X_5, X_9 \in \mathcal B$ are simple
objects of $\mathcal B$ of Frobenius-Perron dimensions $5$ and
$9$, respectively, then $Z \boxtimes X_5$ and $Z \boxtimes X_9$
are simple objects of $\Z(\C) = \D \boxtimes \mathcal B$ of
Frobenius-Perron dimension $5\FPdim Z$ and $9\FPdim Z$. This
contradicts \cite[Theorem 2.11 (i)]{ENO2} because these numbers
cannot both divide $90$. This finishes the proof of the corollary.
\end{proof}

\section{Fusion categories of small dimension}\label{sec3}
The computations in this section were partly handled by a computer. Some of them were done using \cite{GAP}.
For example, it is easy to write a computer program by which one
finds out all possible positive integers $1=d_1,d_2,\cdots,d_s$
and $n_1,n_2,\cdots,n_s$ such that $84=\sum_{i=1}^{s}n_id_i^2$,
and then one can exclude those which are not possible types by
using Proposition \ref{prop1}.

\begin{thm}\label{thm2}
Let $\C$ be an integral fusion category of Frobenius-Perron
dimension $84$. Then $\C$ is of Frobenius type. Further, the group
$\Pic(\C)$ is of order $2$, $3$, $4$, $6$, $12$, $21$ or $28$ and
$\C$ is of one of the following types:
\begin{align*} & (1, 2; 2, 3; 3, 2; 4, 1; 6, 1), (1, 2; 2, 3; 3, 6; 4,
1),  (1, 2; 3, 2; 4, 4),  \\
&   (1, 3; 2, 18; 3, 1), (1, 3; 2, 9; 3, 5), (1, 3; 2, 9; 3, 1; 6,
1),& \\
& (1, 3; 2, 6; 3, 1; 4, 3),   (1, 3; 3, 9),  (1, 3; 3, 5; 6,
1), (1, 3; 3, 1; 6, 2),   & \\
& (1, 4; 2, 20), (1, 4; 2, 16; 4, 1),  (1, 4; 2, 12 ; 4, 2), (1,
4; 2, 8; 4, 3), (1, 4; 2, 4; 4, 4),  & \\ &   (1, 4; 2, 2; 3, 8),
(1, 4; 2, 2; 6, 2), (1, 4; 2, 2; 3, 4; 6, 1), (1, 4; 4, 5), &
  \\
& (1, 6; 2, 3; 3, 2; 4, 3), (1, 6; 2, 6; 3, 2; 6, 1), (1, 6;
2, 6; 3, 6), (1, 6; 2, 15; 3, 2), & \\
& (1, 12; 2, 18), (1, 12; 3, 8), (1, 12; 6, 2), (1, 12; 3,
4; 6, 1), (1, 12; 2, 6; 4, 3), & \\
& (1, 21; 3, 7), (1, 28; 2, 14).& \end{align*}
\end{thm}

\begin{proof}
Assume first that $\C$ is of Frobenius type.  In view of the results in Section
\ref{sec2}, a computer program gives the prescribed ones as the only possible
types in dimension $84$. In fact, using Proposition \ref{prop1}, the list of all possible types can be reduced to
the type $(1,1;3,2;4,1;7,1)$ and the types appearing in the claim.
But the type $(1,1;3,2;4,1;7,1)$ can be discarded as follows: Let $x_3\in
\Irr_3(\C)$. Then $x_3x_3^*=\1+2x_4$, where $x_4\in \Irr_4(\C)$. From
$m(x_4,x_3x_3^*)=m(x_3,x_4x_3)=2$, we have $x_4x_3=2x_3+2x_3'$, where $x_3\neq
x_3'\in \Irr_3(\C)$. Then $m(x_3',x_4x_3)=m(x_4,x_3'x_3^*)=2$, which means that
$x_3'x_3^*=2x_4+\1$. So $x_3=x_3'$, a contradiction.

\medbreak Now assume that $\C$ is not of Frobenius type.  It follows from Proposition
\ref{prop1} that $\C$ has one of the following types: $(1,1;3,1;5,1;7,1)$,
$(1,2;3,2;8,1)$, $(1,4;4,1;8,1)$, $(1,2;4,2;5,2)$.  We shall exclude these
types below.

\medbreak \noindent \textbf{Type  $(1,1;3,1;5,1;7,1)$.} Let
$x_3,x_5,x_7 \in \Irr (\C)$ of Frobenius-Perron dimensions $3,5$
and $7$, respectively. From $x_3x_3^*=\1+x_3+x_5$, we have
$m(x_5,x_3x_3^*)=m(x_3,x_5x_3)=1$, which means that
$x_5x_3=x_3+x_5+x_7$. Then
$m(x_5,x_5x_3)=m(x_5,x_3x_5)=m(x_3,x_5^2)=1$, which means that
$x_5^2=\1+x_3+3x_7$. Then $m(x_7,x_5^2)=m(x_5,x_7x_5)=3$, which
means that $x_7x_5=3x_5+2x_3+2x_7$. Then
$m(x_3,x_7x_5)=m(x_7,x_3x_5)=2$, which means that
$x_3x_5=2x_7+\1$. This contradicts Schur's Lemma.

\medbreak \noindent \textbf{Type  $(1,2;3,2;8,1)$.} Let $x_3 \in
\Irr_3(\C)$. Then
$x_3x_3^*=\1+x_8$, where $x_8$ is the unique simple object  of
Frobenius-Perron dimension $8$. From
$m(x_8,x_3x_3^*)=m(x_3,x_8x_3)=1$, we have
$x_8x_3=x_3+mx_3'+nx_8$, where $3m+8n=21$ and $x_3\neq x_3'\in
\Irr_3(\C)$. From $m(x_3',x_8x_3)=m(x_8,x_3'x_3^*)=m$, we know
that $m\leq1$. Hence, the equation $3m+8n=21$ can not hold true.

\medbreak \noindent \textbf{Type  $(1,4;4,1;8,1)$.} Let $x_4,x_8
\in \Irr(\C)$ of Frobenius-Perron dimensions $4$ and $8$,
respectively, and $\Pic(\C)=\{\1,g_1,g_2,g_3\}$. Then
$$x_4^2 = x_4x_4^*=\1+g_1+g_2+g_3+3x_4 \mbox{\,\,or\,\,}
x_4x_4^*=\1+g_1+g_2+g_3+x_4+x_8.$$ If the first possibility holds,
then the set $\Pic(\C)\cup \Irr_4(\C)$ spans a fusion subring of $K_0(\C)$.
Hence, $\C$ has a fusion subcategory of Frobenius-Perron dimension
$20$. This is impossible since $20$ does not divide $\FPdim \C$.
If the second possibility holds, then
$m(x_8,x_4x_4^*)=m(x_4,x_8x_4)=1$, which means that
$x_8x_4=x_4+mx_8$, where $m$ is a non-negative integer. This is
impossible.

\medbreak \noindent \textbf{Type  $(1,2;4,2;5,2)$.} Let $x_4 \in
\Irr_4(\C)$ and let
$\Pic(\C)=\{1,g\}$. Then there must exist $z \in \Irr_5(\C)$ such
that $m(z,x_4x_4^*)=1,2$ or $3$.

If $m(z,x_4x_4^*)=3$ then $m(x_4,z x_4)=3$. This means that
$z x_4=3x_4+2y$, where $x_4\neq y \in \Irr_4(\C)$. Then
$m(y,z x_4)=m(z,yx_4^*)=2$. This means that $yx_4^*=2z+u$, where
$u$ is an object of dimension $6$ such that $m(z, u) = 0$. Notice that, in our case, $x_4x_4^*=\1+3z$ and
hence $y=gx_4$. Hence, $yx_4^*=gx_4x_4^*=g(\1+3z)$, which
contradicts with $m(z,yx_4^*)=2$.

If $m(z,x_4x_4^*)=2$ then $m(x_4,z x_4)=2$. This means that $z
x_4=2x_4+3y$, where $x_4\neq y\in \Irr_4(\C)$. Then $m(y,z
x_4)=m(z,yx_4^*)=3$. This implies that $yx_4^*=3z+g$, and hence
$y=gx_4$. Hence, $yx_4^*=gx_4x_4^*=3z+g$, which means that
$x_4x_4^*=3gz+\1$. This contradicts the assumption that
$m(z,x_4x_4^*)=2$.

If $m(z,x_4x_4^*)=1$ then $m(x_4,zx_4)=1$. This means that $z
x_4=x_4+4y$, where $x_4\neq y\in \Irr_4(\C)$. Then $m(y,z
x_4)=m(z,yx_4^*)=4$. This is a contradiction. This completes the
proof of the theorem.
\end{proof}

\begin{thm}\label{thm3}
Let $\C$ be an integral fusion category of Frobenius-Perron
dimension $90$. Then $\C$ is of Frobenius type. Moreover, the
group $\Pic(\C)$ is of order $2$, $6$, $9$, $10$, $15$,
$18$, $30$ or $45$ and $\C$ is of one of the following types:
\begin{align*}& (1, 2; 2, 4; 3, 8), (1, 2; 2, 4; 6, 2), (1, 2; 2, 4; 3, 4;  6, 1), (1, 2; 2,
22),&\\
& (1, 6; 2, 3; 3, 8), (1, 6; 2, 3; 6, 2), (1, 6; 2, 3; 3, 4; 6, 1), (1, 6; 2, 21), &\\
& (1, 9; 3, 9), (1, 9; 3, 1; 6, 2), (1, 9; 3, 5; 6, 1), (1, 9; 9, 1), (1, 10; 2, 20), (1, 15; 5, 3),& \\
&
 (1, 18; 2, 18), (1, 18; 3, 8), (1, 18; 3, 4; 6, 1), (1, 18; 6, 2),  (1, 30; 2, 15), (1, 45; 3, 5).&
\end{align*}
\end{thm}

\begin{proof}
Assume first that $\C$ is of Frobenius type. Besides of the types
listed in the claim, a computer calculation  combined with
Proposition \ref{prop1} yields
    the additional types $(1,5;2,10;3,5)$, $(1,9;2,9;3,5)$, $(1,3;2,3;5,3)$,
    $(1,9;2,9;3,1;6,1)$.  These four types can be excluded by Lemma \ref{lem4}.

\medbreak Now assume that $\C$ is not of Frobenius type. By
Proposition \ref{prop1},
    $\C$ has one of the following types:\\
$(1,1;3,1;4,1;8,1)$,
$(1,2;2,2;4,1;8,1)$,
$(1,2;4,1;6,2)$,
$(1,2;3,8;4,1)$, $(1,6;3,4;4,3)$, $(1,2;3,4;4,1;6,1)$,\,\,
$(1,2;2,2;4,5)$,\,\,
$(1,2;2,1;4,3;6,1)$,\,\, $(1,2;2,1;3,2;4,1;5,2)$,\\
$(1,2;2,1;3,4;4,3)$,  $(1,3;2,3;3,3;4,3)$, $(1,6;4,3;6,1)$, $(1,6;2,9;4,3)$.

\medbreak Note that the type $(1,3;2,3;3,3;4,3)$ is discarded by
Lemma \ref{lem4} (i). In addition, since every weakly
group-theoretical fusion category is of Frobenius type, then the
types $(1,6;3,4;4,3)$,   $(1,6;  4,3;6,1)$
and $(1,6;2,9;4,3)$ are discarded by Corollary \ref{90-6}. Similarly, the types $(1,2;2,1;4,3;6,1)$,
$(1,2;2,1;3,2;4,1;5,2)$ and  $(1,2;2,1;3,4;4,3)$ are discarded by Lemma \ref{lem4} (iv) and Corollary \ref{90-6}. We
shall exclude the remaining types below.


\medbreak \noindent \textbf{Type $(1,1;3,1;4,1;8,1)$.} Let
$x_3,x_4, x_8 \in \Irr(\C)$ of Frobenius-Perron dimensions $3$,
$4$ and $8$, respectively. Then $x_3x_3^*=\1+2x_4$ or $x_3x_3*=\1+x_8$. In the first case,
$m(x_4,x_3x_3^*)=m(x_3,x_4x_3)=2$. This means that
$x_4x_3=2x_3+z$, where $\FPdim z=6$ and $z$ is a sum of simple
objects of Frobenius-Perron dimension $4$ or $8$, which is
impossible.
In the second case, $m(x_8,x_3x_3^*)=m(x_3,x_8x_3)=1$.
This means that $x_8x_3=x_3+u$, where $\FPdim u=21$ and $u$ is a
sum of simple objects of Frobenius-Perron dimension $4$ or $8$. It
is also impossible.

\medbreak \noindent \textbf{Type  $(1,2;2,2;4,1;8,1)$.} Let
$x_4 \in \Irr_4(\C)$. Then there
exists $x_2 \in \Irr_2(\C)$ such that
$m(x_2,x_4x_4^*)=1$ or $2$. If $m(x_2,x_4x_4^*)=1$ then
$x_2x_4=x_4+x_2'+x_2''$, where $x_2',x_2''\in \Irr_2(\C)$. Then
$m(x_2',x_2x_4)=m(x_2'^*,x_4x_2^*)=m(x_4,x_2'^*x_2)=1$. This means
that $x_2'\neq x_2''$ and hence $m(x_2,x_2x_4)=1$. Then
$m(x_4,x_2^*x_2)=1$, which is impossible. Therefore
$m(x_2,x_4x_4^*)=2$, for every $x_2\in \Irr_2(\C)$. Then
$x_4x_4^*=\1+g+2x_2+y$, where $\FPdim y=10$ and
$m(x_2,y)=0$. Taking Frobenius-Perron dimensions, we see that this is impossible.

\medbreak \noindent \textbf{Type  $(1,2;4,1;6,2)$.} Let $x_4$ be
the unique element of $\Irr_4(\C)$. Then $x_4x_4^*=\1+g+2x_4+x_6$, where $\{\1,g\}=\Pic(\C)$ and
$x_6\in \Irr_6(\C)$. It follows that
$m(x_6,x_4x_4^*)=m(x_4,x_6x_4)=1$ and $x_6x_4=x_4+z$, where
$m(x_4,z)=0$ and $\FPdim z=20$. It is impossible.

\medbreak \noindent \textbf{Type  $(1,2;3,8;4,1)$.} Let $x_4$ be
the unique element of $\Irr_4(\C)$. Then there must exist an element $x_3$ of $\Irr_3(\C)$ such that
$m(x_3,x_4x_4^*)=m(x_4,x_3x_4)=1$, $2$ or $3$. If
$m(x_4,x_3x_4)=1$ or $2$, then the decomposition of $x_3x_4$ gives
rise to a contradiction. Therefore, $m(x,x_4x_4^*)=3$ or $0$, for
every $x\in \Irr_3(\C)$. Then the decomposition of $x_4x_4^*$ gives
rise to a contradiction.

%
%
%
%

\medbreak \noindent \textbf{Type  $(1,2;3,4;4,1;6,1)$.} Let $x_4
\in \Irr_4(\C)$. If $m(x_3,x_4x_4^*)>0$ for some $x_3\in\Irr_3(\C)$,
then $m(x_4,x_3x_4)=1$, $2$ or $3$. If $m(x_4,x_3x_4)=1$ or $2$,
then the decomposition of $x_3x_4$ gives rise to a contradiction.
If $m(x_4,x_3x_4)=3$ then $x_4x_4^*=\1+g+3x_3+z$, where
$\{\1,g\}=\Pic(\C)$ and $\FPdim z=5$. It is impossible. Therefore,
$x_4x_4^*=\1+g+2x_4+x_6$, where $x_6\in \Irr_6(\C)$. It follows that
$m(x_6,x_4x_4^*)=m(x_4,x_6x_4)=1$ and $x_6x_4=x_4+z$, where
$\FPdim z=20$ and $m(x_4,z)=0$. It is also impossible since $z$ is
the sum of elements of $\Irr_3(\C)$ and $\Irr_6(\C)$.

\medbreak \noindent \textbf{Type  $(1,2;2,2;4,5)$.} Write
$\Pic(\C)=\{\1,g\}$. Because $|\Irr_4(\C)|=5$ is odd, there
exists a self-dual $x_4\in\Irr_4(\C)$. Moreover,
$G[x_4]=\Pic(\C)$. Counting degrees, we find that there exists
only one element of $\Irr_2(\C)$ appearing in the decomposition of
$x_4x_4^*$ with multiplicity $1$, say $x_2$. Clearly, $x_2$ is
self-dual. Let $x_2\neq x_2'\in \Irr_2(\C)$.
Since $G[x_2]=G[x_2']=\Pic(\C)$, Lemma \ref{prod-simples-categorico} shows
that $x_2x_2'$ and $x_2'x_2$ are not irreducible. The fusion rules
of elements of $\Irr(\C)$ show that
$x_2x_2'=x_2'x_2=x_2+x_2'$.

From $m(x_2,x_4^2)=m(x_4,x_2x_4)=1$, we have
$x_2x_4\stackrel{(1)}{=}x_4+x_4'$, where $x_4\neq
x_4'\in \Irr_4(\C)$. Then
$m(x_4',x_2x_4)=m(x_2,x_4'x_4)\stackrel{(2)}{=}1$. Counting
degrees, $x_4'x_4=x_2+x_2'+\omega$ is the only possible
decomposition of $x_4'x_4$, where $\FPdim \omega=12$ and
$m(x_2',\omega)=0$. Then $m(x_2',x_4'x_4)=m(x_4',x_2'x_4)=1$,
which means that $x_2'x_4\stackrel{(3)}{=}x_4'+x_4''$, where
$x_4'\neq x_4''\in\Irr_4(\C)$. From
$m(x_2',x_4'x_4)=m(x_4,x_2'x_4')=1$, we have
$x_2'x_4'=x_4+x_4'''$, where $x_4\neq x_4'''\in\Irr_4(\C)$.
Multiplying equality (1) on the left by $x_2$, we have
$x_2^2x_4=x_2x_4+x_2x_4'$.

If $x_2^2=\1+g+x_2$, then $(\1+g+x_2)x_4=2x_4+x_2x_4=x_2x_4+x_2x_4'$. This shows that
$x_2x_4'=2x_4$. It follows that
$m(x_4,x_2x_4')=m(x_2,x_4x_4'^*)=m(x_2,x_4'x_4)=2$. This
contradicts equality (2).

If $x_2^2=\1+g+x_2'$, then
$(\1+g+x_2')x_4=2x_4+x_2'x_4=2x_4+x_4'+x_4''
=x_2x_4+x_2x_4'=x_4+x_4'+x_2x_4'$. This shows that
$x_2x_4'=x_4+x_4''$. Multiplying on the left by $x_2'$, we have
\begin{eqnarray*}
x_2'x_2x_4'&=&(x_2+x_2')x_4'=x_2x_4'+x_2'x_4'
= x_4+x_4''+x_4+x_4'''\\&=& x_2'x_4+x_2'x_4''
= x_4'+x_4''+x_2'x_4''.
\end{eqnarray*}


This shows that $x_4'+x_2'x_4''=2x_4+x_4'''$. Hence, $x_4'=x_4'''$
and $x_2'x_4''=2x_4$. Then
$m(x_4,x_2'x_4'')=m(x_4,(x_4'')^*x_2')=m((x_4'')^*,x_4x_2')=m(x_4'',x_2'x_4)=2$.
This contradicts equality (3).
\end{proof}

\begin{proof}[Proof of Theorem \ref{fpdim100}] Let $\C$ be a
fusion category and suppose $\FPdim \C = N$ is a natural number
and  $N< 120$. In view of the results of \cite{ENO2}, every
fusion category of Frobenius-Perron dimension $60$, $p^aq^b$ or
$pqr$, where $p$, $q$ and $r$ are prime numbers, $a, b \geq 0$, is
weakly group-theoretical. Indeed, if $\FPdim \C = 60$, this is
shown in \cite[Theorem 9.16]{ENO2}. If $\FPdim \C = p^aq^b$, then
$\C$ is solvable and thus weakly group-theoretical \cite[Theorem
1.6]{ENO2}. If $\FPdim \C = pqr$, then either $\C$ is integral and
thus group-theoretical \cite[Theorem 9.2]{ENO2}, or $\C$ is a
$\mathbb Z_2$-extension of a fusion subcategory $\D$ and we may
assume that $p = 2$ and $\FPdim \D = qr$. Then $\FPdim \D$ is odd
and $\D$ is group-theoretical by \cite{EGO}. Hence $\C$ is weakly
group-theoretical also in this case. In particular, all such
fusion categories have the strong Frobenius property \cite[Theorem
1.5]{ENO2}.

We may therefore assume that $N = 84 = 2^2.3.7$ or $90 = 2.3^2.5$.
If $\C$ is not integral, then $\C$ is a $\mathbb Z_2$-extension of
a fusion subcategory $\D$ \cite[Theorem 3.10]{gel-nik} with
$\FPdim \D = 42$ or $45$, respectively. By the previous discussion, $\D$ is
weakly group-theoretical, and therefore so is $\C$. Hence in this
case $\C$ has the strong Frobenius property and in particular, it
is of Frobenius type. Finally, if  $\C$ is integral, then $\C$ is
of Frobenius type, by Theorems \ref{thm2} and \ref{thm3}.

Suppose next that $\FPdim \C > 1$ and $\C \ncong \rep
\mathbb A_5$.  It follows from \cite[Theorem 1.6]{ENO2}, that if
$\FPdim \C = p^aq^b$, where $p$ and $q$ are prime numbers, then
$\C$ has nontrivial invertible objects (see Lemma \ref{pa-qb}).

If $\FPdim \C = pqr$, where $p$, $q$ and $r$ are distinct prime
numbers, then either $\C$ is a $\mathbb Z_2$-extension of a fusion
subcategory $\D$ or $\C$ is group-theoretical. In the first case,
we know that $\D$ has nontrivial invertible objects, whence so
does $\C$. Thus we may assume that $\C$ is group-theoretical, that
is, $\C$ is equivalent as a fusion category to the category $\C(G,
\omega, H, \psi)$ of $k_{\psi}H$-bimodules in $\C(G, \omega)$,
where $G$ is a group of order $pqr$, $H \subseteq G$ is a
subgroup, $\omega \in Z^3(G, k^*)$ and $\psi \in C^2(H, k^*)$ are
such that $\omega\vert_H = d\psi$. We may assume that $H \neq 1$
(otherwise $\C$ is pointed and we are done). Then the group
$\widehat H$ of linear characters on $H$ is also nontrivial, and
it follows from \cite[Theorem 5.2]{gel-naidu} that $\Pic(\C) \neq
\1$ also in this case.

Suppose that $\FPdim \C = 60$. Since, by assumption, $\C \ncong
\rep \mathbb A_5$,   \cite[Theorem 9.12]{ENO2} implies that $\C$
has a proper fusion subcategory $\D$. The previous discussion
shows that $\Pic (\D) \neq \1$ and therefore also $\Pic(\C) \neq
\1$.

It remains to  consider the cases where $\FPdim \C = 84$ or $90$.
As before, we may assume that $\C$ is integral. The result follows
in this case from Theorems \ref{thm2} and \ref{thm3},
respectively.
\end{proof}



\begin{thebibliography}{0}

\bibitem{Bichon} Bichon J., Natale S., Hopf algebra deformations of binary polyhedral
groups, Transform. Groups \textbf{16}, 339--374 (2011).

\bibitem{dong2}Dong J., Dai L., Further results on semisimple Hopf algebras of dimension
$p^2q^2$, Rev. Uni\' on Mat. Argent. \textbf{53}, 97--112 (2012).

\bibitem{DGNOI} Drinfeld, V., Gelaki, S., Nikshych, D., Ostrik, V., On braided fusion categories
I, Sel. Math. New Ser. \textbf{16}, 1--119 (2010).

\bibitem{EGO} Etingof P., Gelaki S., Ostrik, V., Classification of fusion categories of dimension $pq$, Int.
Math. Res. Not. \textbf{2004} (57), 3041--3056 (2004).

\bibitem{ENO}  Etingof P., Nikshych D., Ostrik V., On fusion categories, Ann. Math. (2) \textbf{162}, 581--642 (2005).

\bibitem{ENO2}Etingof P., Nikshych D., Ostrik V.,
Weakly group-theoretical and solvable fusion categories, Adv.
Math. \textbf{226}(1), 176--505 (2011).

\bibitem{GAP}
  The GAP~Group, \emph{GAP -- Groups, Algorithms, and Programming,
  Version 4.5.5};
  2012,
  \verb+http://www.gap-system.org+.

\bibitem{gel-naidu} Gelaki S., Naidu, D.,
Some properties of group-theoretical categories,
J. Algebra \textbf{322}, 2631-2641 (2009).

\bibitem{gel-nik} Gelaki S., Nikshych, D., Nilpotent fusion categories,
Adv. Math. \textbf{217}, 1053-1071 (2008).

\bibitem{isaacs}Isaacs M., Finite group theory, Graduate Studies in Mathematics \textbf{92}, Amer. Math. Soc., Providence, 2008.

\bibitem{Kaplansky} Kaplansky I., Bialgebras,  University of Chicago Press, Chicago, 1975.

\bibitem{LR}  Larson, R. G., Radford, D. E., Semisimple Hopf algebras,
J. Algebra \textbf{171}, 5-35 (1995).

\bibitem{mueger}  M\" uger, M.,  On the Structure of Modular Categories, Proc. Lond. Math. Soc. III. Ser. \textbf{87}, 291--308 (2003).

\bibitem{Natale4} Natale S.,
Semisolvability of semisimple Hopf algebras of low dimension, Mem.
Amer. Math. Soc. \textbf{186} (874) (2007).

\bibitem{cd2-wint} Natale S., Plavnik, J.,  Solvability of a
class of braided fusion categories, Appl. Categ. Struct. \textbf{22}, 229--240 (2014).

\bibitem{Nichols}Nichols W. D., Richmond M. B., The Grothendieck group of a Hopf
algebra, J. Pure Appl. Algebra \textbf{106}, 297--306 (1996).

\bibitem{szhu}Zhu S., On finite dimensional semisimple Hopf algebras, Commun. Algebra \textbf{21}(11),
3871--3885 (1993).
\end{thebibliography}
\end{document}